\documentclass[12pt]{amsart}
\usepackage{amssymb}
\usepackage{enumerate}
\usepackage{graphicx}
\usepackage{tikz}
\usepackage{parskip}
\usepackage{fancyhdr}
\makeatletter
\@namedef{subjclassname@2010}{%
  \textup{2010} Mathematics Subject Classification}
\makeatother
\newtheorem{thm}{Theorem}[section]
\newtheorem{corollary}[thm]{Corollary}
\newtheorem{lemma}[thm]{Lemma}
\newtheorem{proposition}[thm]{Proposition}
\theoremstyle{definition}

\newtheorem{remark}[thm]{Remark}
\newtheorem{example}[thm]{Example}
\begin{document}
\baselineskip=15pt

\title{Steps in Anderson-Badawi's Conjecture on n-Absorbing and Strongly n-Absorbing Ideals}{}
\author[M. Delic]{Matija Delic}
\address{Department of Mathematics,
Massachusetts Institute of Technology, Cambridge,
Massachusetts, USA}
\email{mdelic@mit.edu }
\author[K. Adarbeh]{Khalid Adarbeh $^{(\star)}$}
\address{Department of Mathematics, An-Najah National University, Nablus, Palestine}
\email{khalid.adarbeh@najah.edu}

\thanks{$^{(\star)}$ Corresponding author}
\date{}

\begin{abstract}
This article aims to solve positively Anderson-Badawi Conjecture of n-Absorbing and strongly n-absorbing ideals of commutative rings in the class of u-rings.  The main result extends and recovers Anderson-Badawi's related result on Prufer domains \cite[Corollary 6.9]{AB}.
\end{abstract}



\maketitle

\section{Introduction}
Throughout this article, $R$ denotes a commutative ring with $1 \neq 0$. In 2007, A. Badawi  introduced the concept of 2-absorbing ideals of commutative rings as a generalization of prime ideals. He defined an ideal $I$ of $R$ to be 2-absorbing if whenever $a, b, c \in R$ and $abc \in I$, then  $ab$ or $ac$ or $bc$ is in $I$ \cite{B1}. As in the case of prime ideals, 2-absorbing have a characterization in terms of ideals. Namely,   $I$  is 2-absorbing if whenever $I_{1}, I_{2}, I_{3}$ are ideals of $R$ and $I_{1}I_{2}I_{3} \subseteq I$, then  $I_{1}I_{2}$ or $I_{1}I_{3}$ or $I_{2}I_{3}$ is contained in $I$ \cite[Theorem 2.13]{B1}.  

In 2011, D.F. Anderson, A. Badawi inspired from the definition of 2-absorbing ideals and defined the n-absobing ideals for any positive integer n. Where an ideal $I$ is called n-absorbing ideal if whenever $x_{1} \dots  x_{n+1} \in I$ for $x_{1}, \dots,  x_{n+1} \in R$  then there are n of the $x_{i}$’s whose product is in I. Also they introduced the strongly-n-absorbing ideals as another generalization of prime ideals, where an ideal $I$ of $R$ is said to be a strongly n-absorbing ideal if whenever $I_{1} \dots · · · I_{n+1} \subseteq  I$ for ideals $I_{1}, \dots, · · ·, I_{n+1}$ of $R$, then the product of some n of the $I_{j}$’s is contained in $I$. Obviously,  a strongly n-
absorbing ideal of $R$ is also an n-absorbing ideal of $R$, and by the last fact in the previous paragraph, 2-absorbing and strongly 2 absorbing are the same. Moreover  D.F. Anderson, A. Badawi  were able to prove that n-absorbing and strongly n-absorbing are equivalent in the class of Prufer domains \cite[Corollary 6.9]{AB}, and they  conjectured that these
two concepts are equivalent in any commutative ring \cite[Conjecture 1]{AB}.  

In 1975, Jr. P. Quartararo  and H.S. Butts defined  the  u-rings to be those rings in which if an ideal $I$ is contained in the union of ideals, then it must be contained in one of them. Then they proved that it suffices to consider the case $I$ is finitely generated ideal of $R$ \cite[Proposition 1.1]{QB}. Moreover, in \cite[Corollary 1.6]{QB}, they proved that the class of Prufer domains (domains in which every finitely generated ideal is invertible) is contained in the class of u-rings. So we have the following diagram of implications:
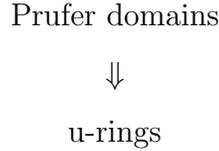
\begin{center}
    Prufer domains
    
    $\Downarrow$
    
    u-rings
\end{center}
where the implication is irreversible in general; see Example \ref{e:1} for a u-ring which is not a domain, particularly, not a Prufer domain.  

In section one of this paper, we provide an alternative proof of \cite[Theorem 2.13]{B1}. The technique of this proof helps in proving the main result of Section 2. In section 2, we solve positively  Anderson-Badawi's Conjecture of n-Absorbing and strongly n-absorbing ideals in the class of u-rings.  The main result (Theorem \ref{T:2})  extends and recovers Anderson-Badawi's related result on Prufer domains (Corollary \ref{c:1}).
\section{Alternative proof}

As we mentioned in the introduction, 2-absorbing ideals and strongly 2-absorbing are the same. This  follows trivially from \cite[Theorem 2.13]{B1}.  In this section, we present an alternative proof of \cite[Theorem 2.13]{B1}, which inspires us in solving \cite[Conjecture 1]{AB} in the class of u-rings. For the seek of completeness, We provide the proof of the following lemma; which can be found as an exercise in the classical ring theory texts.
\begin{lemma} \label{l:1}
Let  $I$ be an ideal of $R$. If  $I=I_1\cup I_2$, where $I_1$ and $I_2$ are also ideals, then $I=I_1$ or $I=I_2$. 
\end{lemma}
\begin{proof}
Suppose $I_1\setminus I_2$ and $I_2\setminus I_1$ are nonempty. Let $a\in I_1\setminus I_2$ and $b\in I_2\setminus I_1$. Since $I_1\cup I_2$ is ideal,  $a+b\in I_1\cup I_2$.  Assume, without loss of generality, that $a+b\in I_1$. Then $b=(a+b)-a\in I_1$, a contradiction. Therefore, either  $I_1\setminus I_2 = \phi$ or $I_2\setminus I_1 = \phi$; equivalently, $I_1 \subseteq I_2$ or $I_2 \subseteq I_1$.
So that $I=I_1$ or $I=I_2$.
\end{proof}
Now, we prove a few lemmas in a sequence, finishing with the proof of the theorem.
\begin{lemma} \label{l:2}
Suppose that $I$ is a 2-absorbing ideal of $R$, $J$ is an ideal of $R$ and $xyJ\subseteq I$ for some $x, y \in R$. Then $xy\in I$ or $xJ\subseteq I$ or $yJ\subseteq I$.
\end{lemma}
\begin{proof}
Suppose $xy\not\in I$. Denote by $J_x=\{z\in J \ |\ xz\in I\}$ and $J_y=\{z\in J\ |\ yz\in I\}$. It is not hard to show that $J_x$ and $J_y$ are ideals. Now, if $a \in J$, then $xya \in I$. But $I$ being 2-absorbing and $xy\not\in I$ imply that $xa \in I$ or $ya \in I$.  Thus, either $a \in J_x$ or $a \in J_y$, and hence $J=J_x\cup J_y$. Therefore, by Lemma \ref{l:1}, either $J=J_x$, and hence $xJ\subseteq I$ or $J=J_y$, and hence $yJ\subseteq I$.
\end{proof}
We generalize the previous lemma as follows.
\begin{lemma} \label{l:3}
Suppose that $I$ is a 2-absorbing ideal of $R$, $I_1$ and $I_2$ are ideals of $R$, and  $xI_1I_2\subseteq I$ for some $x \in R$. Then $xI_1\subseteq I$ or $xI_2\subseteq I$ or $I_1I_2\subseteq I$.
\end{lemma}
\begin{proof}
Suppose $xI_2\not\subseteq I$. By Lemma \ref{l:2}, for all $y\in I_1$, either $xy\in I$ or $yI_2\subseteq I$. Let $N=\{y\in I_1 \ | \ xy\in I\}$ and $M=\{y\in I_1\ |\ yI_2\subseteq I\}$. Then $M$ and $N$ are ideals of $R$, and simlirly as in the proof of Lemma \ref{l:2}, $I_1=N\cup M$. Thus, again  by Lemma \ref{l:1},  either $I_1=N$, and in this case  $xI_1\subseteq I$, or $I_1=M$, and in this case  $I_1I_2\subseteq I$.
\end{proof}
Finally, we use the above lemmas to prove the main theorem.
\begin{thm} \cite[Theorem 2.13]{B1} \label{T:1}
An ideal $I$ of $R$ is 2-absorbing ideal if and only if it is strongly 2-absorbing ideal. 
\end{thm}
\begin{proof}
Obviously, strongly 2-absorbing ideals are 2-absorbing. Conversely, Assume that $I$ is 2-absorbing and $I_1I_2I_3 \subseteq I$, where $I_1$, $I_2$, and $I_3$ are ideals of $R$. Further, Suppose $I_2I_3\not\subseteq I$, and let $N=\{x\in I_1 \ | \ xI_2\subseteq I\}$ and $M=\{x\in I_1 \ |\ xI_3\subseteq I\}$. Then $M$ and $N$ are ideals. By Lemma \ref{l:3}, all $x\in I_1$ are in either $N$ or $M$, and thus $I_1=N\cup M$.
Therefore by Lemma \ref{l:1}, either $I=N$ or $I=M$; which implies that $I_1I_2\subseteq I$ or $I_1I_3\subseteq I$.
\end{proof}

\section{The conjecture}
The following conjecture was announced in \cite{AB}.\\
\textbf{Anderson and Badawi's conjecture:} In every ring, the notions of $n$-absorbing ideals and strongly $n$-absorbing ideals are equivalent.\\

It is easy to see that strongly $n$-absorbing ideals are $n$-absorbing. We aim to find conditions for the converse to hold. We adopt the following terminology from \cite{G} and \cite{QB}: If $I_1,...,I_n$ are ideals of $R$, then   $I_1\cup ...\cup I_n$ is called an efficient covering of $I$ if $I\subseteq I_1\cup ...\cup I_n$, but $I$ is not contained in the union of any $n-1$ of these ideals \cite{G}. In view of this definition, an ideal $I$ of $R$  is called a u-ideal if there is no efficient covering of $I$ with more then one ideal.\\

The following result  solves Anderson and Badawi's conjecture to u-rings,  generalizing thus Corollary 6.9 from \cite{AB}. 
\begin{thm}\label{T:2}
In a $u$-ring, an $n$-absorbing ideal is strongly $n$-absorbing.
\end{thm}
In order to prove this main theorem, we prove the following four lemmas:

\begin{lemma} \label{l:21}
 A principal ideal is a u-ideal.
\end{lemma}
\begin{proof}
Say $I\subseteq I_1\cup...\cup I_n$, and $I=(x)$. Then for some $j$, $x\in I_j$ so $I\subseteq I_j$.
\end{proof}

\begin{lemma}\label{l:22}
Let $I$ be an $n$-absorbing ideal of $R$, and $I_1, ..., I_{n+1}$ be  $u$-ideals of $R$.
Suppose that the following condition is satisfied:\\

    whenever $I_1\cdots I_{n+1}\subseteq I$, and at least $k+1$ of the ideals $I_1,...,I_{n+1}$ are principal, then $I$ contains a product of some $n$ of them.\\
    
Then the same holds when we replace $k+1$ with $k$. Here $n\geq k\geq 0$.
\end{lemma}
\begin{proof}
Assume the statement is true for $I$ and $k+1$. Let $I_1\cdots I_{n+1}\subseteq I$, where $I_j$ is principal for $j\leq k$. Assume $\prod_{j\leq n}I_j \not\subseteq I$. For all $i\leq n$, let
$$J_i=\{y\in I_{n+1}\ |\quad y\prod_{j\neq n+1, i}I_j \subseteq I\}$$ 
Then by our assumption, $I_{n+1}=\cup_{i\leq n}J_i$. Since it is a $u$-ideal, it is equal to some $J_i$. But then 
$$\prod_{j\neq i}I_j\subseteq I$$
This concludes the proof.
\end{proof}

\begin{lemma} \label{l:23}
Let $I$ be an $n$-absorbing ideal. If $I_1\cdots I_{n+1}\subseteq I$, where every $I_j$ is a $u$-ideal, then $I$ contains the product of some $n$ of these ideals.

\end{lemma}

\begin{proof}
By the definition of $I$, and Lemma \ref{l:21}, the statement holds when $I_1, ..., I_{n+1}$ are all principal ideals. We use Lemma \ref{l:22} to induct down from the case $k=n$ (where we require $k+1$ ideals to be principle) to $k=0$ (where we require no ideals to be principle), which is exactly what we want.
\end{proof}

This allows us to prove the main theorem of this article (Theorem \ref{T:2}).

\textbf{Proof of Theorem \ref{T:2}}: Assume the contrary. Then in some $u$-ring, there are ideals $I,I_1, ..., I_{n+1}$ such that $I$ is $n$-absorbing and $I_1\cdots I_{n+1}\subseteq I$, but $I$ doesn't contain the product of any $n$ of these ideals. But $R$ is a u-ring, and hence $I_1, ..., I_n$ are  $u$-ideals. Lemma \ref{l:23} gives a contradiction.

\begin{remark}
We can alter the proof of Lemma \ref{l:23} above slightly, to get a more general statement when $n=2$. Indeed, notice that if $I=I_1\cup I_2$, then $I=I_1$ or $I_2$ (well-known). Then we can drop the condition of the ideals needing to be $u$-ideals from  Lemma \ref{l:23}, and hence we obtain for arbitrary rings, every $2$-absorbing ideal is strongly $2$-absorbing. This is Theorem \ref{T:1}.
\end{remark}

We can use this to give an alternative proof to corollary 6.9 from \cite{AB}. To achieve that, we cite the following results first.
\begin{proposition}
Every invertible ideal is a $u$-ideal, and a Prüfer domain is a $u$-ring.

\end{proposition}
\begin{proof}
See Theorem 1.5 and Corollary 1.6 from \cite{QB}.
\end{proof}
As a straightforward application of Theorem \ref{T:2}, we recover Anderson-Badawi’s related
result on Prufer domains
\begin{corollary}\label{c:1}
In Prüfer domains,  an $n$-absorbing ideal is strongly $n$-absorbing.
\end{corollary}

 Lastly, to ensure that u-rings is strictly larger that the class of  Prüfer domains, we prove the following lemma which provides an example of one such family of u-rings. A more general result, proved in the same way, can be found in \cite{QB}.
\begin{lemma}\label{l:25}
Suppose $R$ is a ring with $\mathbb{Q}\subseteq R$. Then $R$ is a $u$-ring.
\end{lemma}
\begin{proof}
Let $I = I_1\cup \dots\cup I_n$ be an efficient covering of $I$. Take $a_1\in I_1$ with $a_1\not\in I_j$ for $j\neq 1$. Choose $a_2$ analogously. Then for all $k\in \mathbb{Z}$, $a_1+ka_2\not\in I_1, I_2$. Since there are infinite possibilities for $k$, there will be $a_1+ka_2$ and $a_1+la_2$ in the same $I_j$. But then $(k-l)a_2\in I_j$, so $a_2\in I_j$ for $j\neq 2$, contradiction.
\end{proof}
The following is an example of a u-ring which is not a   domain, and hence not a  Prüfer domain. 
\begin{example}\label{e:1}
$\mathbb{Q} \times \mathbb{Q}$ is a ring with zero divisors (not domain) which contains $\mathbb{Q} \cong {0} \times \mathbb{Q}$ as a subring. Consequently, by Lemma \ref{l:25},  $\mathbb{Q} \times \mathbb{Q}$ is a u-ring. 
\end{example}

\textbf{Acknowledgements}. We are grateful to the Undergraduate Research Opportunities Program at MIT (UROP) as well as to  the J-WEL Grant in Higher Education Innovation, “Educational Innovation in Palestine,”  for providing and funding this research opportunity. Also, we would like to thank  Professor
Haynes Miller for his crucial role in  mentoring this project.



\begin{thebibliography}{99}

\normalsize
\baselineskip=17pt

\bibitem{AB}  D.F. Anderson, A. Badawi, On n-absorbing ideals of commutative rings, Comm. Algebra, 39 (2011), 1646-1672.
\bibitem{B1} A. Badawi, On 2-absorbing ideals of commutative rings, Bull. Austral. Math. Soc., 75 (2007), 417-429.
\bibitem{TF} Yousefian Darani, A., Puczylowski, E.R., On 2-absorbing commutative semigroups
and thier applications to rings, Semigroup Forum 86 (2013), 83-91
\bibitem{QB} P. Quartararo Jr. \& H.S. Butts, Finite unions of ideals and modules, Proc. Amer. Math. Soc. 52 (197W5), 91-96.
\bibitem{M} S. McAdam, Finite coverings by ideals, Ring theory, Proceedings of the
Oklahoma conference, Lecture Notes in Pure and Applied Mathematics
volume 7, Dekker 1974, 163-171.
\bibitem{McCoy} N. McCoy, A note on ﬁnite unions of ideals and subgroups, Proc. Amer.Math. Soc. 8 (1957), 633-637
\bibitem{G} Gottlieb, Christian. On finite unions of ideals and cosets. Communications in Algebra. 22 (1994), 3087-3097
\end{thebibliography}
\end{document}